\documentclass[12pt]{article}
\usepackage[final]{epsfig}
\usepackage{graphics}
\usepackage{amsmath}
\usepackage{amsfonts}
\usepackage{latexsym}
\usepackage{amssymb}
\usepackage{graphicx}
\usepackage{url}
\usepackage{epstopdf}
\usepackage{hyperref}
\usepackage{xcolor}
\usepackage{comment}
\usepackage{tikz}
\usepackage{marginnote}
\usetikzlibrary{arrows.meta, positioning}

\newtheorem{lemma}{Lemma}

\newtheorem{proposition}[lemma]{Proposition}

\newtheorem{example}[lemma]{Example}
\newtheorem{theorem}{Theorem}
\newtheorem{definition}[lemma]{Definition}
\newtheorem{corollary}[lemma]{Corollary}

\newcommand{\g}{{\gamma}}
\newcommand{\G}{{\Gamma}}
\renewcommand{\d}{{\delta}}
\newcommand{\D}{{\Delta}}
\newcommand{\eps}{{\varepsilon}}
\newcommand{\proofend}{$\Box$\bigskip}

\newcommand{\R}{{\mathbb R}}
\newcommand{\Z}{{\mathbb Z}}
\newcommand{\RP}{{\mathbb {RP}}}

\def\proof{\paragraph{Proof.}}

\begin{document}

\title{On continuous 2-frieze patterns}

\author{
Serge Tabachnikov\footnote{
Department of Mathematics,
Penn State University, 
University Park, PA 16802;
tabachni@math.psu.edu}
}

\date{}

\maketitle

\begin{abstract}
We define and study a continuous version of 2-frieze patterns, a combinatorial structure closely related with frieze patterns of Coxeter and Conway. We describe the relation of  continuous 2-friezes with the moduli space of projective curves and relate the  (pre)symplectic structure on the space of closed 2-friezes, considered as a cluster variety,   with the Adler-Gelfand-Dikii bracket on the space of 3rd order differential operators.
\end{abstract}

\section{Introduction} \label{sect:int}

The goal of this note is to describe a continuous limit of  discrete combinatorial structures known as 2-frieze patterns, introduced and studied in \cite{MOT}.\footnote{We have learned that John Conway described this structure, that he called ``triezes", to Donald Knuth in a conversation in 1995. Apparently, Conway never developed  these ideas. We thank D. Knuth for sharing this information.}
 Our approach is similar to that of \cite{OT}, where continuous frieze patterns were defined and studied. As a motivation to the present work, we review the material of \cite{OT} in Section \ref{sect:mot} and then  define continuous 2-frieze patterns in Section \ref{sect:c2f}.

Our main results are Theorems \ref{thm:equiv} and \ref{thm:clim} which culminate Sections \ref{sect:1stres} and \ref{sect:symp} respectively. Without providing details here, the first theorem establishes a 1-to-1 correspondence between closed continuous 2-frieze patterns and projective equivalence classes of closed locally convex curves in the projective plane. This is an analog of Theorem \ref{thm:cfrcurve} about continuous frieze patterns which we present in Section \ref{sect:mot}. 

The space of 2-frieze patterns has the structure of a cluster variety, and as such, it is endowed with a canonical (pre)symplectic structure.
The second theorem identifies a continuous limit of this structure with the symplectic form on the modulo space of closed projective curves, associated with the Adler-Gelfand-Dikii bracket on the space of differential operators of degree 3. This is a counterpart to Theorem \ref{thm:cfrsymp} about continuous frieze patterns, also presented in Section \ref{sect:mot}.

It is likely that the results of this note and of \cite{OT} can be generalized to $SL_k$-tilings, see \cite{BR,MOST}. 

\bigskip

{\bf Acknowledgements}. Many thanks to Maxim Arnold, Boris Khesin, and Valentin Ovsienko for useful discussions and  help. 
I was supported by the NSF grant DMS-2404535 and the Simons Foundation grant MPS-TSM-00007747. I am grateful to the Institute of Advanced Studies of the Tel Aviv University, where this paper was written,  for its support and hospitality. I am also grateful to the Israeli Home Front Command for timely warnings about the incoming missiles.

\section{Motivation: continuous frieze patterns} \label{sect:mot}

Frieze patterns (or, simply, friezes) were  introduced by Coxeter and studied by Conway and Coxeter in the early 1970s; the subject has become very popular with the advent of the theory of cluster algebras in the 2000s; see \cite{MG} for a comprehensive survey. Recall the definition.

 A frieze pattern is an array of numbers 
$$
\begin{array}{ccccccccccc}
&0&&0&&0&&0&&\\
&&1&&1&&1&&\\
&a_1&&a_{2}&&a_{3}&&a_{4}&&\\
&&a_1a_2-1&&a_2a_3-1&&a_3a_4-1&&\\
&\cdots&&a_1a_2a_3-a_1-a_3&&a_2a_3a_4-a_2-a_4&&\cdots\\
&&\cdots&&\cdots&&\cdots&&\\
\end{array}
$$
in which the relation
$$
 \begin{matrix}
 &N&
 \\
W&&E
 \\
&S&
 \end{matrix}
 \qquad
 \longmapsto
 \qquad EW-NS=1
$$
holds for every elementary  ``diamond". A frieze is closed if, going down, another row of 1s appear, followed by a row of 0s. The number of rows, bordered by the two rows of 1s, is the frieze pattern's width.

Frieze patterns are discrete combinatorial objects, closely related with difference equations and polygons. A continuous counterpart of a difference equation is a differential equation, and that of a polygon is a curve. This point of view led us to define and study in \cite{OT} continuous analogs of frieze patterns.  Let us briefly describe this work. 

Label the entries of a frieze as
$$
\begin{array}{ccc}
&v_{i,j}&\\
v_{i,j-1}&&v_{i+1,j}\\
&v_{i+1,j-1}&
\end{array}
$$
with $v_{i,i}=a_i$, the entries of the first non-trivial row.
Then the SE diagonals of a closed frieze satisfy the 2nd order linear recurrence 
$V_{i+1} = a_i V_{i} - V_{i-1}$
 with $n$-periodic coefficients, whose solutions are antiperiodic: 
 $V_{i+n}=-V_i,$
that is, the {monodromy} of the equation is -Id. 

Consider the South-East direction as the first, and the North-East as the second coordinate direction, and replace $v_{ij}$ with a smooth function of two variables $H(x,y)$:
$$
\begin{array}{ccc}
&v_{i,j}&\\
v_{i,j-1}&&v_{i+1,j}\\
&v_{i+1,j-1}&
\end{array}
\ \ \longmapsto\ \ 
\begin{array}{ccc}
&H(x,y+\eps)&\\
H(x,y)&&H(x+\eps,y+\eps).\\
&H(x+\eps,y)&
\end{array}
$$
Then the diamond rule yields $\eps^2(H H_{xy} - H_{x} H_{y})=1 +(\eps^3)$, leading to the next definition.

\begin{definition} \label{def:cfrieze}
A {closed continuous frieze pattern} is
a smooth function $F(x,y)$ satisfying:
$$
\begin{array}{rl}
H H_{xy} - H_{x} H_{y}=1;\\
H(x,x) =0,\ H_{y}(x,x)=1,\ H_x(x,x)=-1 &{\rm for\ all}\ x;\\
H(x,y)>0 & {\rm for}\ x<y<x+T;\\ 
H(x+T,y)=H(x,y+T)=-H(x,y)& {\rm for\ all}\ x,y.
\end{array}
$$
\end{definition}

One of the main results of \cite{OT} is the description of continuous friezes in terms of $SL_2$-equivalence classes of projective curves and the respective Hill operators. These are objects of projective differential geometry, see \cite{OTb}. Here are the relevant notions.

A {projective curve} is a diffeomorphism $\g:\R/{\rm T}\Z \to \RP^1.$ 
Projective curves are considered modulo $SL_2(\R)$. One lifts $\g$ to a curve $\Gamma(x)\subset \R^2$,  normalized by the condition $\det(\G,\G')=1$. Then $\G''(x)=-k(x) \G(x)$, that is, the coordinates of $\G$ satisfy the Hill differential equation $y''(x)=-k(x) y(x)$ with $T$-periodic potential $k(x)$ and $T$-anti-periodic solutions: $y(x+T)=-y(x)$. The equivalence classes of projective curves are in 1-1 correspondence with such Hill differential operators $d^2+k(x)$.

\begin{theorem} [\cite{OT}] \label{thm:cfrcurve}
All continuous frieze patterns are obtained from $SL_2$-equivalence classes of convex projective curves:  $H(x,y)=\det(\Gamma(x),\Gamma(y)).$
\end{theorem}

Explicitly, a function $f(x)$ (a projective curve in affine coordinate) yields a continuous frieze
$$
H(x,y)=\frac{f(y)-f(x)}{\sqrt{f'(x)f'(y)}}.
$$
For example, if $f(x)=\tan x$, then $H(x,y)=\sin(y-x)$, and if $f(x)=x$, then $H(x,y)=y-x$.

Another result of \cite{OT}  relates the (pre)symplectic form on the space of frieze patters, associated with its structure of cluster variety with the Kirillov symplectic form on the modulo space of projective curves, the space that can be regarded as a coadjoint action of the Virasoro algebra (see, e.g., \cite{OTb} for this material). One has

\begin{theorem} [\cite{OT}] \label{thm:cfrsymp}
The space ${\rm Diff_+}(S^1)/SL_2(\R)$,  equipped with Kirillov’s symplectic form, is a continuous limit of the space of frieze patterns, equipped with the cluster(pre)symplectic form.
\end{theorem}

Our goal is to obtain analogs of these results for 2-frieze patterns.

\section{Continuous 2-friezes and continuous $SL_3$-tilings} \label{sect:c2f}

A 2-frieze pattern (or just a 2--frieze) is a grid of numbers satisfying the local relations 
\begin{equation} \label{eq:twofr}
\begin{array}{ccccccccccc}
&&&B&F&&&&\\
&&A&E&D&H&&\\
&&&C&G&&&&
\end{array}
\qquad \longmapsto \qquad \ E=AD-BC,\ D=EH-FG,
\end{equation}
see \cite{MOT}.  Note that, unlike frieze patterns, the next rows are not offset relative to the previous ones.
Similalrly to frieze patterns, we label the entries of a 2-frieze as follows:
$$
\begin{array}{ccccccccccc}
&&(i-2,j)&&(i-1,j+1)&&(i,j+2)&&&\\
&&(i-1,j-1)&&(i,j)&&(i+1,j+1)&&&\\
&&(i,j-2)&&(i+1,j-1)&&(i+2,j)&&&
\end{array}
$$

A closed 2-frieze is bounded from above by two rows of zeros, followed by a row of ones, and from below, by a row of ones, followed by two rows of zeros. We refer to it as the upper and lower boundary conditions. 

A related notion is that of $SL_3$-tiling. This is an array of numbers on the same (rotated $45^\circ$) grid as the frieze patterns, 
such that every elementary $3\times 3$ determinant equals 1; see \cite{BR}. An $SL_3$-tiling is called tame if 
every $4\times 4$ determinant of the adjacent numbers vanishes; equivalently, the rank of the infinite matrix formed by this array equals 3. The relation between 2-friezes and $SL_3$-tilings is that the numbers $v_{i,j}$ of a 2-frieze with even index $i$ form an $SL_3$-tiling, and likewise for odd $i$ (Proposition 3.3 in \cite{MOT}).

Define continuous 2-friezes and continuous $SL_3$-tilings. First, the heuristics. 
As before, replace the 2-indexed arrays of numbers  by functions of two variables, and the unit shift of an index by the increase or decrease of variables by an infinitesimal quantity, $2\eps$. This leads to the following consideration. 

We  have two smooth function, $F(x,y)$ and $G(x,y)$. Assume that $E$ in (\ref{eq:twofr}) corresponds to $G(x,y)$. Then the values of its neighbors are 
$$
A=F(x-\eps,y-\eps), B=F(x-\eps,y+\eps), C=F(x+\eps,y-\eps), D=F(x+\eps,y+\eps),
$$
and we obtain the equation
$$
G(x,y)=F(x-\eps,y-\eps)F(x+\eps,y+\eps) - F(x-\eps,y+\eps)F(x+\eps,y-\eps).
$$
Expanding to the second order in $\eps$  and normalizing by dividing by $4\eps^2$ yields
$$
G(x,y)=F(x,y) F_{xy}(x,y)- F_x(x,y)F_y(x,y), 
$$
where the subscripts denote the partial derivatives. One has a similar equation where $F$ and $G$ are swapped. 

Next consider  the upper boundary condition of a 2-frieze. The coordinates of a horizontal row satisfy $y-x=const$, and we  take $F(x,x)=0$ as an analog of the statement that the top row consists of zeros. The row below it also consists of zeros, hence, in the linear approximation, $F(x+2\eps,y)=F(x,y-2\eps)=0$, that is, $F_x(x,x)=F_y(x,x)=0$. The  row below it consists of ones, implying that 
$$
F(x,y-2\eps)=F(x+\eps,y-\eps)=F(x+2\eps,y),
$$
thus, calculating modulo $\eps^3$ and normalizing, we obtain
$$
F_{xx}(x,x)=F_{yy}(x,x)=1, F_{xy}(x,x)=-1.
$$
Likewise for $G$.

We have an analog of the periodicity property of closed 2-friezes along its diagonals:
$$
F(x+T,y)=F(x,y+T)=F(x,y),
$$
where $T$ is a period.  We arrive at the next definition.

 \begin{definition} \label{def:cfr}
A continuous 2-frieze is a couple of smooth functions of two variables, $F(x,y)$ and $G(x,y)$, satisfying the identities
\begin{equation} \label{eq:FtoG}
G=FF_{xy}-F_xF_y,\ F=GG_{xy}-G_xG_y.
\end{equation}
A closed continuous 2-frieze satisfies  the following boundary conditions:
\begin{equation} \label{eq:closed}
\begin{aligned}
F(x,x)=&F_x(x,x)=F_y(x,x)=0, F_{xx}(x,x)=F_{yy}(x,x)=1,\\
&F_{xy}(x,x)=-1, \ F(x+T,y)=F(x,y+T)=F(x,y),
\end{aligned}
\end{equation}
 and likewise for $G$. 
 A closed continuous 2-frieze is positive if $F(x,y)>0$ for $x<y<x+T$, and likewise for $G$.
\end{definition}
Similarly, we have 

\begin{definition}  \label{def:csl}
A continuous $SL_3$-tiling is a smooth function $F(x,y)$ satisfying the identity
\begin{equation} \label{eq:det1}
\det \begin{pmatrix}
F & F_x & F_{xx}\\
F_y & F_{xy} & F_{xxy}\\
F_{yy} & F_{xyy} & F_{xxyy}\\
\end{pmatrix}
=1.
\end{equation}
It is tame if, in addition, the identity holds
$$
\det \begin{pmatrix}
F & F_x & F_{xx} & F_{xxx}\\
F_y & F_{xy} & F_{xxy} & F_{xxxy}\\
F_{yy} & F_{xyy} & F_{xxyy} & F_{xxxyy}\\
F_{yyy} & F_{xyyy} & F_{xxyyy} & F_{xxxyyy}
\end{pmatrix}
=0.
$$
\end{definition}

An example of a closed continuous positive 2-frieze is 
$$
F(x,y)=G(x,y)=1-\cos(x-y).
$$
The same function is an example of a tame continuous $SL_3$-tiling.

\section{Continuos 2-friezes and projective curves} \label{sect:1stres}

The next lemma relates continuous 2-friezes and continuous $SL_3$-tilings. It  is an analog of Proposition 3.3 in \cite{MOT}.

\begin{lemma} \label{lm:til}
1) Let $F$ be a continuous $SL_3$-tiling. Define $G=FF_{xy}-F_xF_y$. Then $F=GG_{xy}-G_xG_y$, that is, $(F,G)$ is a continuous 2-frieze.\\
2) If $(F,G)$ is a closed positive continuous 2-frieze then $F$ and $G$ are continuous $SL_3$-tilings.
\end{lemma}

\proof In one direction,  since $G=FF_{xy}-F_xF_y$, we have
$$
G_x=FF_{xxy}-F_yF_{xx},\ G_y=FF_{xyy}-F_xF_{yy},\ G_{xy}=FF_{xxyy}-F_{xx}F_{yy},
$$
that is, these partial derivatives of $G$ are the $2\time 2$ minors of the matrix (\ref{eq:det1}). 
Therefore, by the Dodgson condensation, 
\begin{equation*} \label{eq:2to3}
F=
F\det 
\begin{pmatrix}
F & F_x & F_{xx}\\
F_y & F_{xy} & F_{xxy}\\
F_{yy} & F_{xyy} & F_{xxyy}\\
\end{pmatrix}
=GG_{xy}-G_xG_y.
\end{equation*}

Conversely, if $(F,G)$ is a closed positive continuous 2-frieze, then 
\begin{equation*} \label{eq:2to3}
F=GG_{xy}-G_xG_y =
F\det 
\begin{pmatrix}
F & F_x & F_{xx}\\
F_y & F_{xy} & F_{xxy}\\
F_{yy} & F_{xyy} & F_{xxyy}\\
\end{pmatrix},
\end{equation*}
and canceling $F$ implies that $F$ is a continuous $SL_3$-tiling. Likewise for $G$.
\proofend

\paragraph{Projective curves and differential operators.} By projective curves we mean parameterized locally convex curves in $\RP^2$, considered up to the natural action of the group $SL_3(\R)$ on the projective plane. Since a non-contractible closed curve in $\RP^2$ has an odd number of inflection points, a closed locally convex curve in $\RP^2$ is  contractible. A projective curve is {\it convex} if every line intersects it  at most twice, multiplicity counted.

Let $\g(x)$ be  a projective curve. We lift it to a curve $\G(x)\subset \R^3$. The lift is not unique, and we fix it by asking that $\det (\G,\G_x,\G_{xx})=1$. If $\g(x)=(f(x)\mathpunct{:}g(x)\mathpunct{:}1)$, then $\G(x)=W^{-\frac{1}{3}} (f(x)\mathpunct{:}g(x)\mathpunct{:}1)$, where
$$
W = \det 
\begin{pmatrix}  
f_x & g_{x}\\
f_{xx} & g_{xx}
\end{pmatrix}   
$$
is the Wronskian.
The lift is considered up to the action of $SL_3(\R)$ on $\R^3$. If $\g$ is closed, so is its lift: this follows from the fact that $\g$ is contractible.

Differentiating $\det (\G,\G_x,\G_{xx})=1$, we see that $\G_{xxx}(x)$ is a linear combination of $\G(x)$ and $\G_x(x)$, that is, the components of the curve $\G(x)$ are independent solutions of the differential equation $\G_{xxx}(x)+q(x)\G_x(x)+r(x)\G(x)$. If the curve is $T$-periodic, so are the functions $q(x)$ and $r(x)$, and all the solutions of this differential equation are $T$-periodic as well. 

Thus, to a locally convex closed projective curve $\g$ there corresponds the differential operator $L=d^3+q(x)d+r(x)$ with periodic coefficients and all periodic solutions, and projectively equivalent curves correspond to the same operator. This correspondence is 1-to-1, see \cite{OTb}. If the curve is convex, the  differential operator is {\it disconjugate}: every solution of  the respective differential equation has at most two zeros per period.

To a projective curve $\g$ there corresponds its dual  $\g^*$, a curve in the dual projective plane $(\RP^2)^*$ which consists of the tangent lines to $\g$. If $\g$ corresponds to a differential operator $L$, then $\g^*$ corresponds to the formally dual operator $L^*$.

Identify 3-space with its dual via a Euclidean structure. The next statement is standard.

\begin{lemma} \label{lm:dual}
The lift of the dual curve is given by $\G^*(x)= \G(x) \times \G_x(x)$.
\end{lemma}

\proof
We need to check two things: that $\G(x) \times \G_x(x)$, considered as covector via the dot product, annihilates Span $(\G(x),\G_x(x))$, and that $\G^*(x)$ also satisfies the unit determinant condition.

The first follows from the fact that $u\times v$ is orthogonal to $u$ and $v$ for any two vectors $u$ and $v$. 
For the second, one has
$$
(\G\times\G')'=\G\times\G'', (\G\times\G')''=\G'\times\G''+\G\times\G'''.
$$
Since $\det(\G,\G',\G'')=1$, it follows that $\det(\G,\G',\G''')=0$, that is, $\G'''$ is a linear combination of $\G$ and $\G'$. Finally,
\begin{equation*}
\begin{aligned}
\det(\G\times\G',(\G\times\G')',&(\G\times\G')'')= 
\det(\G\times\G',\G\times\G'',\G'\times\G''+\G\times\G''')\\
 &=\det(\G\times\G',\G\times\G'',\G'\times\G'')=\det(\G,\G'\G'')=1,
\end{aligned}
\end{equation*}
where we used the vector algebra identities
$$
(u\times v)\times w=(u\cdot w) v - (v\cdot w) u,\ \det(u,v,w)=u\cdot (v\times w).
$$
This concludes the proof.
\proofend

\paragraph{Continuous $SL_3$-tilings and 2-friezes from projective curves.} Let $\g(x) \subset \RP^2$ and $\d(y) \subset (\RP^2)^*$ be projective curves, let $\G(x)$ and $\D(y)$ be their lifts to $\R^3$ and $(\R^3)^*$ respectively, and let $\cdot$ be the pairing of vectors and covectors. We consider pairs of curves up to the natural diagonal action of $SL_3$ on $\R^3 \times (\R^3)^*$.

The next result is an analog of Lemma 4.1 in \cite{MOT}.

\begin{proposition} \label{pr:constr1}
Let $F(x,y)=\G(x) \cdot \D(y)$. Then $F$ is a tame continuous $SL_3$-tiling. Conversely, every tame continuous $SL_3$-tiling is obtained in this way. 
\end{proposition}

\proof
Let $U(x)$ be the $3\times 3$ matrix whose columns are $\G,\G_x,\G_{xx}$ and $V(y)$ be the matrix whose rows are $\D,\D_y,\D_{yy}$. Then $V(y) U(x)$ is the matrix in (\ref{eq:det1}). By the definition of the lift, $\det U(x)=\det V(y)=1$, hence (\ref{eq:det1}) holds. 

Similarly, let $\bar U(x)$ be the $4\times 4$ matrix whose columns are $\G,\G_x,\G_{xx},\G_{xxx}$ and likewise for $\bar V(y)$. One has $\det \bar U(x)=0$, and likewise for $\bar V(y)$. It follows that $\det \bar V(y) \bar U(x)=0$, as needed.

The converse statement follows from a general criterion for a function of two variables $F(x,y)$ to be represented as $\sum_{i=1}^n u_i(x)v_i(y)$; such functions are called {\it separable}. Define the Wronskians 
$$
W_n(F)(x,y)=\det ((\partial^i_x \partial^j_y F)(x,y)),\ i,j=0,\ldots,n.
$$
Assume that $W_{n-1}(F) \neq 0$ everywhere and $W_n(F)=0$ identically. Then there exist functions $u_i(x),v_i(y)$ such that $F(x,y)=\sum_{i=1}^n u_i(x)v_i(y)$, see \cite{GR} (and also \cite{Neu,Neu2}, where a similar result is obtained for functions of more than two variables). 

In our situation, the $3\times 3$ Wronskians are equal to 1, and the $4\times 4$ ones vanish. Then 
$$
\G(x)=(u_1(x),u_2(x),u_3(x)),\ \D(y)=(v_1(y),v_2(y),v_3(y)),
$$
and $F(x,y)=\G(x) \cdot \D(y)$, as needed.
\proofend

It follows that if $G$ is defined by $G=FF_{xy}-F_xF_y$, then $(F,G)$ is a continuous 2-frieze.

\begin{lemma} \label{lm:dual}
One has
$$
G(x,y)=\D^*(y) \cdot \G^*(x).
$$
\end{lemma}

\proof
One has the identity
$$
(a \times b) \cdot (c \times d) = (a \cdot c) (b \cdot d) - (a \cdot d) (b \cdot c).
$$
Since 
$$
\G^*(x)=\G(x) \times \G_x(x),\ \D^*(y)=\D(y) \times \D_y(y),
$$
one obtains
$$
\D(y)^* \cdot \G(x)^* = (\D(y) \times \D_y(y)) \cdot (\G(x) \times \G_x(x)) = F F_{xy}-F_xF_y = G,
$$
as needed.
\proofend

\paragraph{Closed positive continuous 2-friezes.} When does a pair of $T$-periodic projective curves $\g$ and $\d$ define a closed 2-frieze and when is it positive? The answer is given by the next theorem.

\begin{proposition} \label{prop:cldual}
1) The continuous 2-frieze, corresponding to closed projective curves $(\g,\d)$, is closed if and only if $\d=\g^*$.\\
2) This closed continuous 2-frieze is positive if and only if $\g$ is convex.
\end{proposition}

\proof
If the continuous 2-frieze is closed then $\G(x) \cdot \D(x)= \G_x(x) \cdot \D(x) =0$. This implies that 
$\D(x)$ is proportional to $\G^*(x)=\G(x) \times \G_x(x)$, and since $1=\det(\G,\G_x,\G_{xx})=\G_{xx} \cdot (\G\times \G_x)$, the condition $\G_{xx}(x) \cdot \D(x)=1$ holds for $\D=\G^*$.

Assume that $\D=\G^*$, that is, $F(x,y)=\G(x) \cdot (\G(y)\times \G_y(y))$. Arguing as above, we have
$F(x,x)=F_x(x,x)=0$ and $F_{xx}(x,x)=1$. Since the projective duality is an involution, we can interchange $\G(x)$ and $\G^*(y)$ and conclude that $F_y(x,x)=0, F_{yy}(x,x)=1$. Finally, 
$$
F_{xy}(x,x)=\G_x(x) \cdot (\G(x)\times \G_{xx}(x))= \det(\G_x(x),\G(x),\G_{xx}(x))=-1,
$$
as needed.

If $\g$ is convex then it has the following Property T:

 {\it every the tangent line to $\g$  does not have other common points with $\g$.}

\noindent It follows that $F(x,y)=\G(x) \cdot \G^*(y)$ does not vanish for  $x<y<x+T$. To show that $F(x,y)>0$, let $y=x+\eps$. Then $F(x,y)=\frac{\eps^2}{2}F_{xx}(x,x) > 0$.

Conversely, assume that $F(x,y)>0$ for $x<y<x+T$. We need to show that $\g$ is a convex curve. Note that $\g$ has Property T.
 Indeed, if the tangent line at $\g(y)$ contains point $\g(x)$ then 
$0=\G(x) \cdot (\G(y)\times \G_y(y))=F(x,y)$. 

It remains to prove Property T implies that the curve  is convex. Consider the tangent line at $\g(0)$ as the line at infinity. Then the  curve $\g(x),\ x\neq 0$, lies in the affine plane. If it is not convex, its convex hull is different from $\g$, hence the curve has a double tangent line, contradicting Property T. 
\proofend

\begin{corollary} \label{cor:sym}
For a closed continuous 2-frieze, one has $G(x,y)=F(y,x)$.
\end{corollary}

\proof
It follows from Proposition \ref{pr:constr1}, Lemma \ref{lm:dual}, and Proposition \ref{prop:cldual}.
\proofend

We summarize the results obtained so far.

\begin{theorem} \label{thm:equiv}
Closed continuous 2-frieze patterns are in 1-to-1 correspondence with projective equivalence classes of closed projective curves. A closed continuous 2-frieze pattern is positive if and only if the respective projective curve is convex.
\end{theorem}

\begin{example} \label{ex:conic}
{\rm Let the projective curve be a conic. Namely, let
$\g(x)=(a\cos x\mathpunct{:}b\sin x\mathpunct{:}1)$. Then
$$
\G(x)=(ab)^{-\frac13} (a\cos x, b\sin x, 1),\ \G^*(y)=(ab)^\frac13 \left(-\frac{\cos y}{a},-\frac{\sin y}{b},1\right),
$$
and 
$$
F(x,y)=\G(x)\cdot \G^*(y)=1-\cos(x-y) = G(x,y)
$$
with $T=2\pi$. We encountered this example earlier.
}
\end{example}

\begin{example} \label{ex:powers}
{\rm The next example satisfies only the upper boundary conditions, and the continuous 2-frieze is not periodic. We omit the somewhat tedious calculations and just state the result.

Let $a,b,c$ be a triple of numbers satisfying $a+b+c=3$, and let
$$
\G(x)=\left(\frac{x^a}{b-c}, \frac{x^b}{c-a}, \frac{x^c}{a-b} \right).
$$
One checks that $\det(\G,\G_x,\G_{xx})=1$. One has
$$
\G^*(y)=\left(\frac{(c-b)y^{2-a}}{(a-b)(c-a)}, \frac{(a-c)y^{2-b}}{(b-c)(a-b)}, \frac{(b-a)y^{2-c}}{(c-a)(b-c)} \right),
$$
hence 
$$
F(x,y)=\frac{x^a y^{2-a}}{(a-b)(a-c)} + \frac{x^b y^{2-b}}{(b-c)(b-a)} + \frac{x^c y^{2-c}}{(c-a)(c-b)}. 
$$
The quadratic case is when $(a,b,c)=(2,1,0)$, then 
$$
F(x,y)=\frac{x^2}{2}-xy+\frac{y^2}{2}.
$$
}
\end{example}

\paragraph{Self-dual closed continuous 2-friezes.} We ask when $F(x,y)=G(x,y)$ for all $x,y$. 

\begin{lemma}  \label{lm:sdual}
This happens if and only if the the projective curve $\g$ is a conic.
\end{lemma}

\proof In one direction, this statement follows from Example \ref{ex:conic}. 

Assume that $F(x,y)=G(x,y)$ for all $x,y$. Then, according to Lemma \ref{lm:dual} and Proposition \ref{prop:cldual},  $\G(x)\cdot \G^*(y) = \G(y)\cdot \G^*(x)$ for  all $x,y$, that is,
\begin{equation} \label{eq:sdual}
\det (\G(x),\G(y),\G_y(y))=\det(\G(y),\G(x),\G_x(x)).
\end{equation}

Recall that one has $\G_{xxx}=k(x)\G_x(x)+v(x)\G(x)$ for some functions $k(x)$ and $v(x)$. Let $h(x)=v(x)-\frac12 k'(x)$. The differential 1-form $h(x)^{\frac{1}{3}} dx$ is called the projective length element on the projective curve. This 1-form vanishes if and only if the curve is a conic (see, e.g., \cite{GMO,OTb}).

Let $y=x+\eps$ in (\ref{eq:sdual}). A calculation up to the 5th order in $\eps$ (which we omit) reveals that $h(x)=0$ identically, and we are done. 
\proofend

\paragraph{Explicit formulas.} Closed continuous 2-friezes can be parameterized by a pair of periodic functions of one variable. 

Let $(f(x),g(x))$ be a closed locally convex curve in $\R^2 \subset \RP^2$. As we know, the lift is given by 
$\G(x)=W(x)^{-\frac{1}{3}} (f(x),g(x),1)$ where
$$
W(x)=\det \begin{pmatrix}
f'(x) & g'(x) \\
f''(x) & g''(x) 
\end{pmatrix} > 0.
$$
Then
\begin{equation*}
\begin{aligned}
F(x,y)&=\det (\G(x),\G(y),\G_y(y))\\
&= W(x)^{-\frac{1}{3}} W(y)^{-\frac{2}{3}} 
\left[ 
\det \begin{pmatrix}
f(y) & g(y) \\
f'(y) & g'(y) 
\end{pmatrix} -
\det \begin{pmatrix}
f(x) & g(x) \\
f'(y) & g'(y) 
\end{pmatrix} 
         \right].
\end{aligned}
\end{equation*}
Also, by Corollary \ref{cor:sym}, one has $G(x,y)=F(y,x)$.

\paragraph{From continuous 2-friezes to continuous friezes.} Recall the notion of {\it projective curvature} of a locally convex curve in $\RP^2$, see \cite{OTb}, Section 1.4. 

As we know, a (projective equivalence class of a) projective curve $\g(x)$  corresponds to a 3rd-order differential operator $d^3+q(x)d+r(x)$. Assign to it the 2nd-order differential operator $d^2+\frac{1}{4}q(x)$. This operator corresponds to a projective equivalence class of a curve $\Psi(x)$ in $\R^2$ satisfying $\Psi''(x)=-\frac{1}{4} q(x) \Psi(x)$. The projection of this curve to $\RP^1$ is a parameterized curve $\psi(x)$ therein, that is, a projective structure on $\RP^1$. The correspondence
\begin{equation} \label{eq:difop3to2}
d^3+q(x)d+r(x)\  \ \longmapsto\  \ d^2+\frac{1}{4}q(x)
\end{equation}
is independent of the choice of the parameter $x$, i.e., it commutes with the action of the group of diffeomorphisms on differential operators, see Theorem 1.4.3 in \cite{OTb}.

Let us describe this correspondence at the level of continuous friezes. Start with a curve $\G(x) \subset \R^3$ satisfying $\det(\G,\G',\G'')=1$ and define a continuous $SL_3$-tiling by 
\begin{equation} \label{eq:FG}
F(x,y)= \G(x) \cdot \G^*(y) = \det (\G(x),\G(y),\G_y(y))
\end{equation} 
(see  Proposition \ref{prop:cldual}).
The curve $\G$ satisfies the differential equation 
\begin{equation} \label{eq:Gd}
\G'''(x)+q(x)\G'(x)+r(x)\G(x)=0.
\end{equation}
Hence
$$
\det(\G,\G'',\G''')=-q,\ \det(\G',\G'',\G''')=r.
$$
Differentiating (\ref{eq:FG}) and combining with (\ref{eq:Gd}), yields
\begin{equation} \label{eq:Fd}
F_{xxx}+q F_x+rF=0.
\end{equation}
On the other hand,  since 
\begin{equation*} 
\det \begin{pmatrix}
F & F_x & F_{xx}\\
F_y & F_{xy} & F_{xxy}\\
F_{yy} & F_{xyy} & F_{xxyy}\\
\end{pmatrix}
=1,
\end{equation*}
one has
\begin{equation*} 
\det \begin{pmatrix}
F & F_x & F_{xxx}\\
F_y & F_{xy} & F_{xxxy}\\
F_{yy} & F_{xyy} & F_{xxxyy}
\end{pmatrix}
=0.
\end{equation*}
Hence the columns of this matrix satisfy the same relation (\ref{eq:Fd}), that is, we may take $\G(x)=(F(x,y),F_y(x,y),F_{yy}(x,y))$ with a fixed value of $y$. Solving the equations, linear in $q$ and $r$,
$$
F_{xxx}+qF_x+rF= F_{xxxy}+qF_{xy}+rF_y=0
$$
yields
$$
-q=\frac{\det \begin{pmatrix}
F & F_{xxx}\\
F_y & F_{xxxy}
\end{pmatrix}}
{\det \begin{pmatrix}
F & F_{x}\\
F_y & F_{xy}
\end{pmatrix}}.
$$ 

A similar consideration applies to a continuous frieze $H(x,y)$. We have a curve $\Psi(x) \subset \R^2$, satisfying the differential equation $\Psi''(x)+\frac{1}{4}q(x)\Psi(x)=0$. Hence 
$$
\det(\Psi',\Psi'')=-\frac{1}{4}q(x).
$$
One has $H(x,y)=\det(\Psi(x),\Psi(y))$ and $\Psi(x)=(H(x,y),H_y(x,y))$ with fixed $y$. Therefore
$$
q=-4 \frac{H_{xx}}{H}.
$$

We arrive at

\begin{lemma} \label{lm:frfrom2fr}
The  continuous frieze corresponding to the continuous $SL_3$-tiling via the correspondence $(\ref{eq:difop3to2})$
is given by the differential equation on $H(x,y)$:
$$
\frac{H_{xx}}{H} = \frac{1}{4} \frac{\det \begin{pmatrix}
F & F_{xxx}\\
F_y & F_{xxxy}
\end{pmatrix}}
{\det \begin{pmatrix}
F & F_{x}\\
F_y & F_{xy}
\end{pmatrix}}.
$$
\end{lemma}

Let us illustrate this by revising Example \ref{ex:powers}.

\begin{example} \label{ex:posers2}
{\rm  Recall that
$$
F(x,y)=\frac{x^a y^{2-a}}{(a-b)(a-c)} + \frac{x^b y^{2-b}}{(b-c)(b-a)} + \frac{x^c y^{2-c}}{(c-a)(c-b)}. 
$$
and that 
$$
\G(x)=\left(\frac{x^a}{b-c}, \frac{x^b}{c-a}, \frac{x^c}{a-b} \right),
$$
where $a+b+c=3$, and let us assume that $ab+bc+ca > 1$ (this holds in the quadratic case $(2,1,0)$). Then a calculation, that we omit, yields
$$
-q=\det(\G,\G'',\G''') = \frac{ab+bc+ca-2}{x^2},
$$
hence
$$
\frac{H_{xx}}{H} = \frac{ab+bc+ca-2}{4 x^2}.
$$
The general solution of this differential equation is
$$
H(x,y)=C_1(y) x^{\alpha_1} + C_2(y) x^{\alpha_2},
$$
where 
$$
\alpha_2=\frac{1 + \sqrt{ab+bc+ca-1}}{2},\ \alpha_1=\frac{1 - \sqrt{ab+bc+ca-1}}{2},\ \ {\rm and} \ \ \alpha_1+\alpha_2=1.
$$

Consider the curve $\psi(x)=(x^{\alpha_1}:x^{\alpha_2}) \subset \RP^1$ and its lift 
$$
\Psi(x)=\frac{1}{\sqrt{\alpha_2-\alpha_1}} (x^{\alpha_1},x^{\alpha_2})
$$ 
to $\R^2$, satisfying $\det(\Psi,\Psi')=1$. Let 
$$
\Phi(y) = \sqrt{\alpha_2-\alpha_1} (-C_2(y),C_1(y))
$$
be another such curve, satisfying $\det(\Phi,\Phi')=1$. Then $H(x,y) = \det (\Psi(x),\Phi(y))$, and the Ptolemy-Pl\"ucker relation
$$
\det(\Psi,\Phi) \det(\Psi',\Phi') - \det(\Psi',\Phi) \det(\Psi,\Phi') = \det(\Psi,\Psi') \det(\Phi,\Phi') =1
$$
on the four vectors $\Psi(x), \Psi'(x), \Phi(y), \Phi'(y)$ implies that 
$$
\det \begin{pmatrix}
H & H_{x}\\
H_y & H_{xy}
\end{pmatrix} =1.
$$

If, in addition, $\Psi=\Phi$, then $H(x,x)=0, H_y(x,x)=1, H_x(x,x)=-1$, which are the upper boundary conditions on this continuous frieze pattern
$$
H(x,y)=\frac{x^{\alpha_1}y^{\alpha_2} - x^{\alpha_2}y^{\alpha_1}}{\alpha_2-\alpha_1}.
$$
}
\end{example}

\section{Symplectic form}  \label{sect:symp}

\paragraph{Symplectic structure on the space of projective curves.} Let us ``unwrap" the formula given in Section 7 of \cite{GO} (also see \cite{Ov}) for the symplectic structure on the space of closes convex projective curves, associated with the Adler-Gelfand-Dikii Poisson bracket on the space of differential operators  \cite{Ad,GD}. This formula is based on the Kupershmidt-Wilson theorem \cite{KW}, describing the ADG Poisson bracket; we use its presentation in Section 10.7 of \cite{BBT}. 
 
Let $\g \subset \RP^2$ be a convex closed projective curve, $\G \subset \R^3$ its lift, and $L$ the associated 3rd order disconjugate differential operator. Let $y_1(x),y_2(x),y_3(x)$ be a basis of solutions of the equation $L(y)=0$ such that  
$$
y_1>0,\ W=\det \begin{pmatrix}
y_1 & y_2\\
y_1' & y_2'
\end{pmatrix} > 0,
$$
where prime is $d/dx$. That such a basis exists follows from disconjugacy of $L$. We think of $y_1,y_2,y_3$ as the components of the curve $\G$.

Let $u(x)=(u_1,u_2,u_3)$ be a vector field along $\G(x)$, an infinitesimal deformation of this curve. To factor out the action of $SL_3$, we  assume that $u(0)=u'(0)=u''(0)=0$. Denote by $u(y_1)$ and $u(W)$ the directional derivatives of $y_1$ and of $W$ along $u$.  

Let\footnote{The matrix involved in the description of the ADG bracket is  
$$
\begin{pmatrix}
2 & -1\\
-1 & 2
\end{pmatrix},
$$
which is, up to a factor, the inverse of $A$. We take the inverse because we consider the symplectic structure, rather than the Poisson bracket.
}

$$
A=\begin{pmatrix}
2 & 1\\
1 & 2
\end{pmatrix},
$$
and let $u(x), v(x)$ be two vector fields along $\G$, considered as tangent vectors to the space of equivalent classes of projective curves at point $\g$. Set
$$
a_1= -\frac{y_1'}{y_1},\ a_2=\frac{y_1'}{y_1} - \frac{W'}{W}.
$$
Then, up to a multiplicative constant, the formula for the symplectic structure on the space of closes convex projective curves reads
\begin{equation*} \label{eq:GO}
\omega(u,v)= \int_{S^1} (d^{-1} u(a_1),d^{-1} u(a_2)) A\, (v(a_1),v(a_2))^T\ dx
\end{equation*}
(the $n=3$ case of Theorem 7.3 in \cite{GO}).\footnote{Correcting a sign error in this formula.}

Let us calculate $u(y_1'/y_1)$ and $u(W'/W)$. 

\begin{lemma} \label{lm:dirder}
One has
$$
u\left(\frac{y_1'}{y_1}\right)=\left(\frac{u_1}{y_1}\right)',\ 
u\left(\frac{W'}{W}\right)= \left(\frac{\det\begin{pmatrix}
u_1 & u_2\\
y_1' & y_2'
\end{pmatrix} -
\det \begin{pmatrix}
u_1' & u_2'\\
y_1 & y_2
\end{pmatrix}}{W}\right)'.
$$
\end{lemma}

\proof We have
$$
u\left(\frac{y_1'}{y_1}\right)= \frac{u(y_1')y_1-u(y_1)y_1'}{y_1^2}=\frac{u_1'y_1-u_1y_1'}{y_1^2}=\left(\frac{u_1}{y_1}\right)',
$$
which proves the first formula.

For the second one, we have
\begin{equation*}
\begin{aligned}
&u(W)=\det\begin{pmatrix}
u_1 & u_2\\
y_1' & y_2'
\end{pmatrix} -
\det\begin{pmatrix}
u_1' & u_2'\\
y_1 & y_2
\end{pmatrix},\\
&u(W') = \det\begin{pmatrix}
u_1 & u_2\\
y_1'' & y_2''
\end{pmatrix} -
\det\begin{pmatrix}
u_1'' & u_2''\\
y_1 & y_2
\end{pmatrix}=
(u(W))',
\end{aligned}
\end{equation*}
hence
\begin{equation*}
\begin{aligned}
u\left(\frac{W'}{W}\right)&=\frac{u(W') W - u(W) W'}{W^2} = \left(\frac{u(W)}{W}\right)' \\
&= 
\left(\frac{\det\begin{pmatrix}
u_1 & u_2\\
y_1' & y_2'
\end{pmatrix} -
\det \begin{pmatrix}
u_1' & u_2'\\
y_1 & y_2
\end{pmatrix}}{W}\right)',
\end{aligned}
\end{equation*}
as claimed.
\proofend

Let us summarize this calculation. Set
\begin{equation*}
\begin{aligned}
&A_u := u(\ln y_1)= \frac{u_1}{y_1},\\
 &B_u := u(\ln\left(\det \begin{pmatrix}
y_1 & y_2\\
y_1' & y_2'
\end{pmatrix}  \right)) = 
\frac{\det\begin{pmatrix}
u_1 & u_2\\
y_1' & y_2'
\end{pmatrix} -
\det \begin{pmatrix}
u_1' & u_2'\\
y_1 & y_2
\end{pmatrix}}{\det \begin{pmatrix}
y_1 & y_2\\
y_1' & y_2'
\end{pmatrix}},
\end{aligned}
\end{equation*}
and similarly for $A_v, B_v$. We obtain

\begin{lemma} \label{lm:om}
One has
$$
\omega(u,v)= \int_{S^1} (2A_uA_v' -  B_u A_v' -  A_u B_v' + 2 B_u B_v')\ dx.
$$
\end{lemma}

Next, let us relate the functions $y_1, y_2$ to the closed continuous 2-frieze $F(x,y) = \det(\G(x), \G(y), \G'(y))$. Fix $y=0$ and assume (acting by $SL_3$ if need be) that 
$$
\G(0)=(0,0,1),\ \G'(0)=(0,-1,0),\ \G''(0)=(1,0,0).
$$
Then
$$
F(x,0)=y_1,\ G(x,0)=\det(\G(x),\G'(x),\G(0))=\det \begin{pmatrix}
y_1 & y_2\\
y_1' & y_2'
\end{pmatrix} = W,
$$
hence $F_x(x,0)=y_1', G_x(x,0)=W'$.

Therefore 
$$
A_u=\frac{dF}{F}(u), A_u'=\frac{FdF_x-F_xdF}{F^2}(u), 
B_u=\frac{dG}{G}(u), B_u'=\frac{GdG_x-G_xdG}{G^2}(u),
$$
and likewise for $v$. This implies  
\begin{equation} \label{eq:om}
\begin{aligned}
\omega(u,v)&= \int_{S^1}  (A_uA_v'-A_vA_u')+(A_u'B_v-A_v'B_u)+(B_uB_v'-B_u'B_v)  \ dx\\
&=\int_{S^1}  \frac{dF}{F} \wedge \left(\frac{dF_x}{F} -\frac{F_xdF}{F^2}  \right)
+  \left(\frac{dF_x}{F} -\frac{F_xdF}{F^2}  \right) \wedge \frac{dG}{G} \\
&+ \frac{dG}{G} \wedge \left( \frac{dG_x}{G} - \frac{G_xdG}{G^2}   \right)  \ dx\\
&= \int_{S^1} \frac{dF\wedge dF_x }{F^2} 
+ \left( \frac{dF_x\wedge dG}{FG}-\frac{F_x dF \wedge dG}{F^2G} \right) + \frac{dG \wedge dG_x}{G^2}\ dx.
\end{aligned}
\end{equation}

\paragraph{Cluster symplectic structure on 2-friezes.} 
Choose, as cluster coordinates, two adjacent South-East diagonals whose entries are $x_i$ and $y_i$. The related closed 2-form is
\begin{equation} \label{eq:form}
\omega = \sum \left( \frac{dx_{i+1}\wedge dx_i}{x_ix_{i+1}} + \frac{dy_{i+1}\wedge dy_i}{y_iy_{i+1}}
+  \frac{dx_{i}\wedge dy_i}{x_iy_i} +  \frac{dy_{i}\wedge dx_{i+1}}{x_{i+1}y_i}\right),
\end{equation}
see Sections 5.3 and 5.5 of \cite{MOT}. 

Now we take continuous limit of (\ref{eq:form} in which the sum is replaced by an integral. Using the same lattice as before, with coordinates step of  $2\eps$, we replace the variable as follows:
$$
x_i \mapsto F(x-\eps,y-\eps), x_{i+1} \mapsto F(x+\eps,y-\eps), y_i \mapsto G(x,y), y_{i+1} \mapsto G(x+2\eps,y).
$$
Under this substitution, the constant term in (\ref{eq:form}) vanishes. Calculating modulo $\eps^2$, we find  the linear in $\eps$ terms: 
\begin{equation*}
\begin{aligned}
\frac{dx_{i+1}\wedge dx_i}{x_ix_{i+1}} \mapsto \frac{(dF + \eps(dF_x-dF_y)\wedge (dF - \eps(dF_x+dF_y)}{F(x-\eps,y-\eps)F(x+\eps,y-\eps)} &\mapsto 
2\frac{dF_x \wedge dF}{F^2},\\ 
\frac{dy_{i+1}\wedge dy_i}{y_iy_{i+1}} \mapsto \frac{(dG +2\eps dG_x)\wedge dG}{G(x,y)G(x+2\eps,y)} &\mapsto  2\frac{dG_x \wedge dG}{G^2},\\
\frac{dx_{i}\wedge dy_i}{x_iy_i} +  \frac{dy_{i}\wedge dx_{i+1}}{x_{i+1}y_i} \mapsto 
\frac{(dF-\eps(dF_x+dF_y))\wedge dG}{F-\eps(F_x+F_y)G} &\\
+ \frac{dG\wedge (dF + \eps(dF_f-dF_y)}{F + \eps(F_x-F_y)G} 
\mapsto 2 \frac{dG \wedge dF_x}{FG} &+ 2 \frac{F_x dF \wedge dG}{F^2G}.
\end{aligned}
\end{equation*}
Thus (\ref{eq:form}) becomes   (\ref{eq:om}), multiplied by a factor of $-2$. We have proved

\begin{theorem} \label{thm:clim}
Up to a multiplicative factor, the continuous limit of the cluster symplectic form on 2-friezes is the Adler-Gelfand-Dikii symplectic form on the modulo space of closed convex projective curves.
\end{theorem}

\end{document}